\documentclass[12pt, 12pt]{elsarticle}
\usepackage{amsmath, amsthm, amscd, amsfonts, amssymb, graphicx, color, mathrsfs}
\usepackage[bookmarksnumbered, colorlinks, plainpages]{hyperref}
\hypersetup{colorlinks=true,linkcolor=red, anchorcolor=green, citecolor=cyan, urlcolor=red, filecolor=magenta, pdftoolbar=true}
\usepackage{color,soul}

\newtheorem{theorem}{Theorem}[section]
\newtheorem{lemma}[theorem]{Lemma}
\newtheorem{proposition}[theorem]{Proposition}

\theoremstyle{definition}
\newtheorem{definition}[theorem]{Definition}
\newtheorem{example}[theorem]{Example}

\theoremstyle{remark}
\newtheorem{remark}[theorem]{Remark}
\numberwithin{equation}{section}

\DeclareMathAlphabet\EuScript{U}{eus}{m}{n}
\DeclareMathAlphabet\EuScriptBold{U}{eus}{b}{n}
\DeclareMathAlphabet\Eurm{U}{eur}{m}{n}
\DeclareMathAlphabet\Eurb{U}{eur}{b}{n}
\newcommand{\mathcalb}{\EuScriptBold}
\newcommand{\AAA}{ {\mathscr A} }

\newcommand{\WWW}{\Omega}

\newcommand{\iii}{\infty}

\newcommand{\negj}{\;\!\!}

 \newcommand{\zza}[1]{\left\vert{#1}\right\vert}   
 \newcommand{\zzao}[1]{\vert{#1}\vert}
 \newcommand{\zzaj}[1]{\bigl\vert{#1}\bigr\vert}
 
 \newcommand{\zzat}[1]{\biggl\vert{#1}\biggr\vert}
 \newcommand{\zzac}[1]{\Biggl\vert{#1}\Biggr\vert}

 \newcommand{\zzmo}[1]{({#1})}
 \newcommand{\zzmj}[1]{\bigl({#1}\bigr)}
 \newcommand{\zzmd}[1]{\Bigl({#1}\Bigr)}
 \newcommand{\zzmt}[1]{\biggl({#1}\biggr)}
 
 \newcommand{\nno}[1]{\zza{\!\!\;\zza{\:\!#1\!\:}\!\!\;}}
 \newcommand{\nnoo}[1]{\zzao{\negj\!\!\;\zzao{\negj\:\!#1\!\:\negj}\!\!\;\negj}}
 \newcommand{\nnoj}[1]{\zzaj{\negj\!\!\;\zzaj{\negj\:\!#1\:\!\negj}\!\!\;\negj}}

\newcommand{\bbdef}{\begin{definition}\sl }
\newcommand{\eedef}{\end{definition}\rm }
\newcommand{\bbl}{\begin{lemma}}
\newcommand{\eel}{\end{lemma}}

\newcommand{\eqd}{\stackrel{def}{=}}
\newcommand{\AAt}{ {\mathscr{A}}_t }

\newcommand{\HH}{{\mathcal H}}
\newcommand{\BH}{ {\mathcalb B}({\mathcal H})}
\newcommand{\ccc}{ {\mathcalb C}}

\newcommand{\ccq}{ {\mathcalb C}_q({\mathcal H}) }

\newcommand{\dt}{\,d\mu(t)}
\newcommand{\LJ}{ L_1(\WWW,\mu)}

\newcommand{\inN}{{\in\mathbb N}}

\newcommand{\BB}{\mathcalb B}

\newcommand{\MM}{\mathfrak{M}}

\newcommand{\GG}{\mathbb{G}}

\newcommand{\befz}{\mathcalb{B}(E,F^*)}

\newcommand{\lpwE}{\ell^w_p(E)}

\newcommand{\Dj}{\mbox{\rule[.75ex]{.4em}{.7pt}\kern-.4em D}}
\newcommand{\djm}{\mbox{\kern.1em\rule[1.2ex]{.3em}{.7pt}\kern-.4em d}}

\journal{Bulletin of Iranian Mathematical Society}

\begin{document}

\begin{frontmatter}



\title{Gel'fand integration of $\BB(E,F^*)$-valued functions with emphasis on $(q,p)$-summing operators}


\author{Matija Milovi\'c}\ead{matija.milovic@matf.bg.ac.rs}
\author{Stefan Milo\v sevi\'c}\ead{stefan.milosevic@matf.bg.ac.rs}

\address{University of Belgrade, Department of  Mathematics, Studentski trg 16,
P.O.box 550, 11000 Belgrade, Serbia}

\begin{abstract}
We generalize results concerning Gel'fand integration of functions taking values in the space of operators on Hilbert spaces to certain Banach spaces. Building on ideas from \cite{M24} we provide sufficient conditions for the Gel'fand integral to be $(q,p)-$summing and we use the developed techniques to answer a question posed in the mentioned article. Applications to positive operator-valued functions between certain function spaces are also given.
\end{abstract}


\end{frontmatter}

\section{Introduction}

Throughout the paper, $E$ and $F$ will always denote complex Banach spaces with norms $\nnoo{\cdot}_E$ and $\nnoo{\cdot}_F,$ respectively, unless otherwise stated. By $E^*,$ we will denote the topological dual of the space $E$ and the action of $y\in E^*$ on $x\in E$ will be denoted by $\langle y,x\rangle.$ By an operator $A\colon E\to F,$ we will always mean a bounded linear operator with its operator norm denoted by $\|A\|,$ and the space of all such operators will be denoted by $\BB(E,F).$ In the case when $E$ and $F$ are Banach lattices, an operator $T\colon E\to F$ will be called positive and denoted by $T\geqslant 0$ if $Tx\geqslant 0$ for every $x\in E$ such that $x\geqslant 0.$ Also, by $\ell_p(E)$ and $\lpwE$ (for $p\geqslant 1$), we will denote Banach spaces of $p$-summable sequences in $E$ and weakly $p$-summable sequences in $E,$ respectively. Banach space $\ell_p(\mathbb{C})$ will be denoted by $\ell_p.$ If there is no risk of ambiguity, we will use $\nnoo{\cdot}_p$ and $\nnoo{\cdot}_{p,w}$ to denote norms in $\ell_p(E)$ and $\lpwE.$ Moreover, for $p\geqslant 1,$ we denote by $p'$ the conjugate exponent, i. e., the number satisfying $\frac{1}{p}+\frac{1}{p'} = 1.$ As we will need some results from the duality of sequence spaces, we state the following lemma for the comfort of the reader. Details can be found in \cite[Proposition~1.3.3]{HNVW}.
\begin{theorem}\label{duallpban}
For $1\leqslant p < +\infty$ the following duality holds $\zzmj{\ell_p(E)}^* \cong \ell_{p'}(E^*)$ in the sense that every $\Phi\in\zzmj{\ell_p(E)}^*$ is of the form $\Phi((x_n)_{n=1}^\infty) = \sum_{n=1}^\infty \langle\phi_n,x_n\rangle,$ for some $(\phi_n)_{n=1}^\iii\in\ell_{p'}(E^*)$ such that $\nnoo{\Phi} = \nnoo{(\phi_n)_{n=1}^\iii}_{p'}.$
\end{theorem}

As a direct consequence, we get the following
\bbl\label{UBP} Let $(\phi_n)_{n=1}^\iii$ be a sequence in a dual Banach space $E^*$ and $1\leqslant p<+\iii$. Then, $(\phi_n)_{n=1}^\iii\in\ell_{p'}(E^*)$ if and only if for every $(x_n)_{n=1}^\iii\in\ell_p(E)$ holds  $\sum_{n=1}^{\iii}|\langle\phi_n,x_n\rangle|<+\iii.$
\eel
\begin{proof}
The proof of one implication follows directly for Holder's inequality, while we prove the converse by contradiction using the Uniform Boundedness Principle and Theorem \ref{duallpban}.
\end{proof}
For a norm $\alpha$ on an algebraic tensor product $E\otimes F,$ by $E\widehat\otimes_\alpha F$ we will denote the completion of the tensor product under that norm. By a $\varepsilon,\pi$ and $\Delta_p$ we will denote, respectively, the injective norm, the projective norm, and the natural norm on the $p$-integrable functions (see Chapter 7 in \cite{DF}). We note that $\ell_p (E) = \ell_p\widehat{\otimes}_{\Delta_p}E,$ while $\ell_p\widehat{\otimes}_\varepsilon E$ is isomorphic to a subspace of $\lpwE$ (proofs and more details can be found in sections 8 and 9 of \cite{DF}). Also, we will make use of the duality
\begin{equation}\label{dualnostBEF}
\zzmj{E\widehat{\otimes}_\pi F}^* \cong \befz,
\end{equation}
proof of which can be found in \cite[Proposition~3.2]{DF}.
\bbdef
An operator $T\colon E\to F$ is said to be $(q,p)$-summing, for $1\leqslant p\leqslant q\leqslant+\iii,$ if there exists $C>0$ such that
$$\zzmt{\sum_{k=1}^n \nnoj{T x_k}_F^q}^{1/q} \leqslant C\cdot \zzmt{\sup_{\|y\|_{E^*}\leqslant1} \sum_{k=1}^n |\langle y,x_k\rangle|^p}^{1/p},$$
for every  $x_1,\ldots,x_n\in E.$ The least $C$ such that the previous inequality holds will be denoted by $\pi_{q,p}(T)$ and the class of all $(q,p)$-summing operators will be denoted by $\Pi_{q,p}(E,F).$
\eedef

Before moving on, we note that in the case $q = +\infty$ a straightforward calculation shows that every $T\in\BB(E,F)$ is $(\iii,p)$-summable with $\pi_{\iii,p}(T)=\nnoo{T}.$ So we will only deal with the case $q<+\iii,$ although most of the statements remain valid considering that $\Pi_{\iii,p}(E,F)=\BB(E,F)$ isometrically.

If an operator $T\colon E\to F$ is $(q,p)$-summing and $(x_n)_{n=1}^\iii\in\lpwE$ by the previous definition, the sequence $(Tx_n)_{n=1}^\iii$ belongs to $\ell_q(F).$ Therefore, each $(q,p)$-summing operator induces an operator $\widehat{T}\colon\lpwE\to\ell_q(F)$ defined by $\widehat{T}((x_n)_{n=1}^\infty) = (Tx_n)_{n=1}^\infty$ and it is straightforward to see that $\|\widehat{T}\| = \pi_{q,p}(T).$ Also, since finite sequences belong to $\lpwE,$ if the operator $\widehat{T}$ from $\lpwE$ to $\ell_q(F)$ is defined in a previously described manner, the operator $T$ is $(q,p)$-summing. Moreover, by the Closed Graph Theorem, it can be shown that $T$ is $(q,p)$-summing if and only if $\widehat{T}(\lpwE)\subset\ell_q(F)$ (details can be found in \cite[Proposition~2.1]{DJT}). Furthermore, due to the connection of the injective tensor norm and $\lpwE,$ the previous definition can also be reformulated as (details can be found in \cite[Proposition~11.1]{DF} and \cite[Exercise~11.20]{DF}).
\bbl\label{qpekv1}
An operator $T\colon E\to F$ is $(q,p)$-summing if and only if the operator
$$id\otimes T\colon\ell_p\otimes_{\varepsilon} E \to \ell_q\otimes_{\Delta_q}F$$
is continuous, where $id$ denotes the canonical embedding of $\ell_p$ into $\ell_q$ for $1\leqslant p\leqslant q<+\iii.$
\eel
For the comfort of the reader, we provide the following lemma that will be used in the sequel.
\bbl\label{qpekv2}
An operator $T\colon E\to F^*$ is $(q,p)$-summing if and only if for every $(x_n)_{n=1}^\infty\in\lpwE$ and $(y_n)_{n=1}^\infty\in\ell_{q'}(F)$ holds
$\sum_{n=1}^\infty \zzao{\langle Tx_n,y_n\rangle} < +\infty.$
In that case we have the following formula
\begin{equation}\label{qpekv2pom}
\pi_{q,p}(T) = \sup_{\substack{\|(x_n)_{n=1}^\infty\|_{p,w}\leqslant 1\\ \|(y_n)_{n=1}^\infty\|_{q'}\leqslant 1}} \sum_{n=1}^\infty \zzao{\langle Tx_n,y_n\rangle}.
\end{equation}
\eel
\begin{proof}
If $T$ is $(q,p)$-summing, as already noted, $(Tx_n)_{n=1}^\iii\in\ell_q (F^*)$ for every $(x_n)_{n=1}^\iii\in\lpwE$ and $\nnoo{(Tx_n)_{n=1}^\iii}_q \leqslant\pi_{q,p}(T)\cdot\nnoo{(x_n)_{n=1}^\iii}_{p,w}.$ Hence
\begin{multline*}
\sum_{n=1}^\infty \zzao{\langle Tx_n,y_n\rangle} \leqslant \sum_{n=1}^\infty \nnoj{Tx_n}_{F^*}\cdot\nnoj{y_n}_F\\
\leqslant \nnoj{\zzmj{T x_n}_{n=1}^\infty}_q\cdot\nnoj{\zzmj{y_n}_{n=1}^\infty}_{q'} \leqslant \pi_{q,p}(T)\cdot \nnoj{\zzmj{x_n}_{n=1}^\infty}_{p,w}\cdot \nnoj{\zzmj{y_n}_{n=1}^\infty}_{q'}<+\iii.
\end{multline*}

Conversely, let $(x_n)_{n=1}^\iii\in\ell_p^w(E)$ be arbitrary. Due to the assumption, by applying Lemma \ref{UBP} to the sequence $(Tx_n)_{n=1}^\iii,$ we get that it belongs to $\ell_q(F^*),$ implying proclaimed $(q,p)$-summability of the operator $T.$ 

If the assumptions hold, for every $(x_n)_{n=1}^\iii$ in the closed unit ball of in $\ell^w_p(E)$ we get
$$\nnoj{\zzmj{Tx_n}_{n=1}^\infty}_q=\sup_{\|(y_n)_{n=1}^\infty\|_{q'}\leqslant 1}\zzac{\sum_{n=1}^\infty \langle Tx_n,y_n\rangle}=\sup_{\|(y_n)_{n=1}^\infty\|_{q'}\leqslant 1} \sum_{n=1}^\infty \zzao{\langle Tx_n,y_n\rangle},$$
due to Theorem \ref{duallpban}. Finally,
$$\pi_{q,p}(T)=\sup_{\|(x_n)_{n=1}^\iii\|_{p,w}\leqslant1}\nnoj{\zzmj{Tx_n}_{n=1}^\infty}_q=\sup_{\substack{\|(x_n)_{n=1}^\infty\|_{p,w}\leqslant 1\\ \|(y_n)_{n=1}^\infty\|_{q'}\leqslant 1}} \sum_{n=1}^\infty \zzao{\langle Tx_n,y_n\rangle},$$
completing the proof.
\end{proof}
The next lemma provides a wide class of $(q,p)$-summing operators in one particular case. Namely,
\bbl\label{pozneplp}
Let $K$ be a compact space and $(X,\nu)$ be an arbitrary measure space. Then every positive operator $T\in\BB(C(K),L_q(X,\nu))$ is $(q,p)$-summing with $\pi_{q,p}(T) = \nno{T}.$
\eel
\begin{proof}
By the previous lemma \ref{qpekv1}, the result is a consequence of \cite[Theorem~7.3]{DF}. See also \cite[Corollary~11.2]{DF}.
\end{proof}
In the case $q=p$, a $(p,p)$-summing operator will be called $p$-summing, and $\Pi_p(E,F)$ and $\pi_p(T)$ will denote the ideal of $p$-summing operators and $p$-summing norm of the operator $T,$ respectively. An important feature that we will use in this paper is that in this case, the space $\Pi_p(E,F^*)$ has a Banach space predual. Namely, for $1\leqslant p<+\infty,$ it holds
\begin{equation}\label{predpsum}
\zzmd{E\widehat{\otimes}_{d_{p'}}F}^* = \Pi_p (E,F^*),
\end{equation}
where $d_p$ is Chevet - Saphar tensor norm given by
$$d_p(z) = \inf\left\{\nnoj{(x_j)_{j=1}^n}_{p',w}\cdot\nnoj{(y_j)_{j=1}^n}_p : z = \sum_{j=1}^n x_j\otimes y_j\right\},\quad\text{for every}\quad z\in E\otimes F.$$
Proofs and more details can be found in sections 6.2. and 6.3. of \cite{R}.

\begin{remark}
    We note that for certain pairs of $q$ and $p,$ the ideal $\Pi_{q,p}(E,F)$ can coincide with the space of all bounded operators $\BB(E,F),$ even for infinite-dimensional Banach spaces. Specially, in the case of the lemma \ref{pozneplp}, for $q>2$ and $1\leqslant p < q$ every bounded operator $T\colon C(K)\to L_q(X,\nu)$ is $(q,p)$-summing (see \cite[Corollary~10.10]{DJT}). Therefore, the mentioned lemma provides us with useful conclusions when $p=q.$
    
    Also, for $1\leqslant q\leqslant2,$ every bounded operator $T\colon C(K)\to L_q(X,\nu)$ is $2$-summing (see \cite[Theorem~3.5]{DJT}), and the result cannot be improved.
\end{remark}

For a measure space $(\WWW,\mathfrak{M},\mu)$ (notation that will also be used in the sequel), a function $f\colon\WWW\to E^*$ will be called weakly$^*$ measurable (integrable) if for every $x\in E$ the scalar function $t\to\langle f(t),x\rangle$ is measurable (integrable). In the case of weakly$^*$ integrable function, due to \cite[Lemma~3.1]{DU}, for every $E\in\mathfrak{M}$ there exists $x_E^*\in E^*$ such that for every $x\in E$ holds $\langle x_E^*,x \rangle = \int_\WWW\langle f(t),x\rangle d\mu(t).$ We will denote $x_E^*$ by $\int_E f(t) d\mu(t)$ and is also called the Gel'fand integral or the weak$^*$ integral of $f$ over $E.$ Likewise, a weakly$^*$ integrable function will also be called Gel'fand integrable. Furthermore, we say that two functions $f,g\colon\WWW\to E^*$ are weakly$^*$ equivalent if for every $x\in E$ holds $\langle f(t),x\rangle = \langle g(t),x\rangle$ almost everywhere, and by $\GG(\WWW,\mathfrak{M},\mu, E^*)$ we denote the space of classes of weakly$^*$ equivalent Gel'fand integrable functions $f\colon\WWW\to E^*$ with norm defined by
$$\nnoj{f}_\GG = \sup_{\nnoo{x}_E\leqslant 1} \int_\WWW \zzaj{\langle f(t),x\rangle} d\mu(t).$$
In this paper, our primary focus will be on functions that take values in $\befz$ which we will also refer to as operator valued (o.v.) functions. Since this space is a dual of a Banach space, as already noted in \eqref{dualnostBEF}, we can define the Gel'fand integral for $\befz$-valued functions. Similarly to \cite[Lemma~1.1]{KMM}, the conditions for measurability and integrability for $\befz$-valued functions can be simplified. More precisely, the following holds.
\begin{lemma}\label{merljivost}
An operator-valued function $\AAA\colon\WWW\to\befz$ is weakly$^*$ measurable (integrable) if and only if for every $x\in E$ and $y\in F$ scalar functions $t\mapsto\langle\AAt x, y\rangle$ are measurable (integrable). 
\end{lemma}
\begin{proof}
Let $\AAA$ be weakly$^*$ measurable (integrable) and $x\in E$ and $y\in F$ be arbitrary. As $x\otimes y\in E\widehat{\otimes}_\pi F,$ the function $t\mapsto\langle\AAt,x\otimes y\rangle=\langle\AAt x,y\rangle$ is measurable (integrable).

Conversely, let the scalar functions $t\mapsto\langle\AAt x,y\rangle$ be measurable, for all $x\in E$ and $y\in F.$ Every $z\in E\widehat{\otimes}_\pi F$ can be written in the form $\sum_{n=1}^\infty x_n\otimes y_n,$ for some sequences $(x_n)_{n\inN}$ in $E$ and $(y_n)_{n\inN}$ in $F$ such that $\sum_{n=1}^\infty \nnoj{x_n}_E\cdot\nnoj{y_n}_F < +\infty.$ It follows that the series $\langle\AAt,z\rangle=\sum_{n=1}^\infty \langle \AAt x_n,y_n\rangle$ absolutely converges implying weak$^*$ measurability of $\AAA,$ since all summands are measurable functions. If, in addition, all functions $t\mapsto\langle\AAt x,y\rangle$ are integrable, we define a bilinear mapping $B\colon E\times F\to L_1(\WWW,\mu)$ by $B(x,y)(t)\eqd\langle\AAt x,y\rangle,$ for $t\in\WWW.$ To prove that $B$ is bounded, it is sufficient to prove that it is separately continuous, by the \cite[1.2.]{DF}. For $y\in F$ we define a linear operator $B_y\colon E\rightarrow\LJ$ by $B_y(x)=B(x,y).$ We prove that $B_y$ is continuous using the Closed Graph Theorem. Let $\|x_n-x\|_E+\|B_y(x_n)-\varphi\|_{\LJ}\rightarrow0,$ as $n\rightarrow\infty,$ for some $\varphi\in\LJ,$ and $x,x_n\in E,$ for $n\inN.$ Then, there is a subsequence $(x_{n_k})_{k\inN}$ such that $B_y(x_{n_k})(t)=\langle\AAt x_{n_k},y\rangle\rightarrow\varphi(t),$ as $k\rightarrow\infty$ for $\mu$ a.e. $t\in\WWW.$ For such $t\in\WWW,$ as $\AAt\in\mathcalb{B}(E,F),$ we have $B_y(x_{n_k})(t)=\langle\AAt x_{n_k},y\rangle\rightarrow\langle\AAt x,y\rangle,$ as $k\rightarrow\infty,$ implying $B_y(x)(t)=\varphi(t)$ for $\mu$ a.e. $t\in\WWW.$ Thus, $B_y$ is a bounded linear operator. In the same fashion, we prove that for every $x\in E$ the linear operator $B_x\colon F\rightarrow\LJ$ defined by $B_x(y)=B(x,y)$ is also continuous, completing the proof.
\end{proof}

\begin{remark}
    In \cite{AK25}, the authors introduce the space of pointwise Dunford integrable functions, denoted by $L_{pD}(\WWW,\MM,\mu,\BB(E,F)),$ where $F$ is not required to be a dual space. The conditions defining this space (see Definition 2.8, and also Proposition 2.9 in the mentioned paper) are similar to, though stronger than, those in the previous lemma, and they apply to $\BB(E,F),$ for all Banach spaces $E$ and $F.$ We also note that in the case where $F$ is reflexive, the spaces $L_{pD}(\WWW,\MM,\mu,\BB(E,F))$ and $\GG(\WWW,\MM,\mu,\BB(E,F))$ are the same, according to Lemma \ref{merljivost}.
\end{remark}

For the o.v. function $\AAA\colon\WWW\rightarrow\BB(E,F),$ we define an o.v. function $\AAA^*\colon\WWW\rightarrow\BB(F^*,E^*)$ by $(\AAA^*)_t=\AAt^*,$ for $t\in\WWW.$ The next proposition deals with the weak$^*$ measurability and integrability of $\AAA^*.$ 

\begin{proposition} Let $\AAA\colon\WWW\rightarrow\BB(E,F^*)$ be weakly$^*$ measurable function, where $F$ is a reflexive Banach space. Then the o.v. function $\AAA^*\colon\WWW\rightarrow\BB(F^{**},E^*)$ is also weak$^*$ measurable. If, in addition, $\AAA$ is weakly$^*$ integrable, so is $\AAA^*,$ and for every $S\in\MM$ holds $(\int_S\AAA d\mu)^*=\int_S\AAA^*d\mu.$
\end{proposition}

\begin{proof}
    Let $x\in E$ and $y^{**}\in F^{**}$ be arbitrary. Since $F$ is reflexive, there exists (unique) $y\in F,$ such that for every $y^*\in F^*$ holds $\langle y^{**},y^*\rangle=\langle y^*,y\rangle.$ Now, we have
    \begin{equation*}
    \langle\AAt^*y^{**},x\rangle=\langle y^{**},\AAt x\rangle=\langle\AAt x,y\rangle.
    \end{equation*}
    Thus, a weak$^*$ measurability (integrability) of $\AAA$ implies a weak$^*$ measurability (integrability) of $\AAA^*.$ In the case of integrability, for $S\in\MM$ we have
    \begin{multline*}
        \left\langle\left(\int_S\AAA^* d\mu\right)y^{**},x\right\rangle=\int_S\langle\AAt^*y^{**},x\rangle d\mu(t)=\int_S\langle y^{**},\AAt x\rangle d\mu(t)\\
        =\int_S\langle\AAt x,y\rangle d\mu(t)=\left\langle\left(\int_S\AAA d\mu\right)x,y\right\rangle=\left\langle y^{**},\left(\int_S\AAA d\mu\right)x\right\rangle
    \end{multline*}
    proving that $(\int_S\AAA d\mu)^*=\int_S\AAA^*d\mu,$ as proclaimed. 
\end{proof}

\begin{remark}
    If the Banach space E is additionally assumed to be reflexive, the results align with those established in \cite[Proposition~2.26]{AK25}.
\end{remark}
In Lemma \ref{merljivost}, we concluded that the weak$^*$ measurability and integrability of an o.v. function $\AAA$ are entirely determined by its associated bilinear form. A similar result holds for evaluating $\|\AAA\|_\GG,$ as shown in the Hilbert space case (see \cite[Proposition~1.8]{M24}).

\begin{proposition}
Let $\AAA\colon\WWW\rightarrow\befz$ be weakly$^*$ integrable. Then, the following holds
\begin{equation}\label{geljfnorma}
\|\AAA\|_{\GG}=\sup_{\|x\|_E=\|y\|_F = 1}\int_{\WWW}|\langle\AAt x,y\rangle|d\mu(t).
\end{equation}
\end{proposition}
\begin{proof}
As for unit vectors $x\in E$ and $y\in F$ holds $\|x\otimes y\|_{E\widehat{\otimes}_\pi F}=\|x\|_E\cdot\|y\|_F=1,$ we have
\begin{equation*}
\int_{\WWW}|\langle\AAt x,y\rangle|d\mu(t)= \int_\WWW|\langle\AAt,x\otimes y\rangle|d\mu(t)\leqslant\sup_{\|T\|_{E\widehat{\otimes}_\pi F}=1}\int_\WWW|\langle\AAt,T\rangle|d\mu(t)=\|\AAA\|_\GG.
\end{equation*}
By taking the supremum over all unit vectors $x$ and $y$ we get the "$\geqslant$" in \eqref{geljfnorma}.

To prove the converse inequality, let $T\in E\widehat{\otimes}_\pi F$ be such that $\|T\|_{E\widehat{\otimes}_\pi F}=1$ and $\varepsilon>0.$ Then, we can find sequences $(x_n)_{n\inN}$ in $E$ and $(y_n)_{n\inN}$ in $F$ of nonzero vectors, such that $T = \sum_{n=1}^\iii x_n\otimes y_n$ and $\sum_{n=1}^{+\iii}\|x_n\|_E\cdot\|y_n\|_F\leqslant 1+\varepsilon.$ Now,
\begin{align*}
&\int_\WWW|\langle\AAt,T\rangle|d\mu(t)=\int_\WWW\left|\sum_{n=1}^{+\iii}\langle\AAt x_n,y_n\rangle\right|d\mu(t)\leqslant\sum_{n=1}^{+\iii}\int_\WWW|\langle\AAt x_n,y_n\rangle|d\mu(t)\\
=&\sum_{n=1}^{+\iii}\|x_n\|_E\cdot\|y_n\|_F\cdot\int_\WWW\left|\left\langle\AAt \frac{x_n}{\|x_n\|_E},\frac{y_n}{\|y_n\|_F}\right\rangle\right|d\mu(t)\\
\leqslant&(1+\varepsilon)\cdot\sup_{\|x\|_E=\|y\|_F=1}\int_{\WWW}|\langle\AAt x,y\rangle|d\mu(t).
\end{align*}
By letting $\varepsilon\downarrow0$ and taking the supremum over $\|T\|_{E\widehat{\otimes}_\pi F}=1$ we get the "$\leqslant$" in \eqref{geljfnorma}, completing the proof. 
\end{proof}
For a Banach Lattice $E$ its dual space $E^*$ has a natural order given by $y\geqslant 0$ iff $\langle y,x\rangle\geqslant 0$ for every positive $x\in E.$ Moreover, $E^*$ is also a Banach lattice. Also, if $E$ and $F$ are Banach lattices, an operator $T\colon E\to F^*$ is positive (by definition) if $Tx\geqslant 0$ for $x\geqslant 0$ and by previous discussion $Tx\geqslant 0$ if and only if $\langle Tx,y\rangle \geqslant$ for every $0\leqslant y\in F.$ Thus, $T\colon E\to F^*$ is positive if and only if $\langle Tx,y \rangle\geqslant 0$ for every positive $x\in E$ and $y\in F.$ Final conclusion provides us with the following
\bbl\label{intpoz}
Let $E$ and $F$ be Banach lattices and $\AAA\colon\WWW\to\befz$ weakly$^*$ integrable function such that $\AAt\geqslant 0$ for almost every $t\in\WWW.$ Then, for every measurable $S\subset\WWW,$ holds $\int_S\AAA d\mu \geqslant 0.$
\eel
\begin{proof}
First, we note that $\langle\AAt x,y\rangle\geqslant 0$ a.e. for positive $x\in E$ and $y\in F$ as $\AAt$ are positive operators a. e. Hence, by the definition of the Gel'fand integral we have
$$\left\langle\zzmt{\int_S \AAA d\mu} x,y\right\rangle = \int_S\left\langle\AAt x,y\right\rangle d\mu(t) \geqslant 0,$$
for positive $x\in E$ and $y\in F.$ Based on the discussion preceding the Lemma, the operator $\int_S \AAA d\mu$ is positive.
\end{proof}

\section{Main results}

We begin by presenting some sufficient conditions under which the Gel'fand integral of an operator-valued function is $(q,p)$-summing.
\begin{theorem}\label{bilin}
Let $\AAA\colon\WWW\rightarrow\befz$ be weakly$^*$ measurable function such that
\begin{equation}\label{bilinuslov}
\sum_{n=1}^\iii|\langle\AAt x_n, y_n\rangle|\in L_1(\WWW,\MM,\mu), \text{for every $(x_n)_{n=1}^\infty\in\lpwE$ and $(y_n)_{n=1}^\infty\in\ell_{q'}(F)$.}
\end{equation}
Then, for every $S\in\MM$ the operator $\int_S\AAA d\mu$ is $(q,p)$-summing and
\begin{equation}\label{konacno}
\pi_{q,p}\left(\int_S\AAA d\mu\right)\leqslant\sup_{\substack{\|(x_n)_{n=1}^\infty\|_{p,w}\leqslant 1\\ \|(y_n)_{n=1}^\infty\|_{q'}\leqslant 1}} \int_S\sum_{n=1}^\iii |\langle\AAt x_n,y_n\rangle|d\mu(t)<+\iii.
\end{equation}
\end{theorem}
\begin{proof}
First, we provide the weak$^*$ integrability of $\AAA.$ Indeed, for every $x\in E$ and $y\in F,$ the sequences $(x,0,0,\ldots)$ and $(y,0,0,\ldots)$ belong to $\lpwE$ and $(y,0,0,\ldots)\in\ell_{q'}(F),$ respectively. Hence, by \eqref{bilinuslov} the function $t\mapsto\langle \AAt x,y \rangle$ belongs to $\LJ$ and therefore, by Lemma \ref{merljivost}, the function $\AAA$ is weakly$^*$ integrable. Specially, for every measurable $S\subset\WWW,$ the operator $\int_S\AAA d\mu$ is defined.

To prove that the operator $\int_S\AAA d\mu$ is $(q,p)$-summing, we use Lemma \ref{qpekv2}. For $(x_n)_{n=1}^\infty\in\lpwE$ and $(y_n)_{n=1}^\infty\in\ell_{q'}(F)$ holds
\begin{multline*}
\sum_{n=1}^\infty \zzac{\left\langle \zzmt{\int_S \AAA d\mu} x_n ,y_n\right\rangle} = \sum_{n=1}^\infty \zzat{\int_S \langle\AAt x_n ,y_n\rangle d\mu(t)}\\
\leqslant \sum_{n=1}^\infty \int_S \zzaj{\langle\AAt x_n ,y_n\rangle} d\mu(t) < +\infty,
\end{multline*}
where last estimate follows by \eqref{bilinuslov}. Hence, by the mentioned lemma, the operator $\int_S\AAA d\mu$ is $(q,p)$-summing, and moreover, by taking the supremum over $\|(x_n)_{n=1}^\infty\|_{p,w}\leqslant 1$ and $\|(y_n)_{n=1}^\infty\|_{q'}\leqslant 1$ we get the first inequality in \eqref{konacno}. The proof that the supremum in \eqref{konacno} is finite is similar to the proof of \eqref{Finormeuslov1} in the sequel and will be omitted.
\end{proof}

The previous Theorem can be further improved if $p=q.$ Namely, in that case, Banach space $\Pi_{q,p}(E,F^*)$ has a predual, and the following holds.
\begin{theorem}\label{Gintuslov}
Let $\AAA\colon\WWW\to\befz$ be weakly$^*$ measurable such that $\AAt\in\Pi_p(E,F^*),$ for $1\leqslant p<+\infty$ and almost every $t\in\WWW$ and the condition \eqref{bilinuslov} holds. Then $\AAA\in\GG(\WWW,\mathfrak{M},\mu,\Pi_p(E,F^*))$ and the following estimate holds
\begin{equation}\label{geljfnorm}
\nnoj{\AAA}_\GG \leqslant \sup_{\substack{\|(x_n)_{n=1}^\infty\|_{p,w}\leqslant 1\\ \|(y_n)_{n=1}^\infty\|_{p'}\leqslant 1}} \int_\WWW\sum_{n=1}^\iii |\langle\AAt x_n,y_n\rangle|d\mu(t) 
\end{equation}
\end{theorem}
\begin{proof}
Due to the duality mentioned in \eqref{predpsum}, all we have left is to prove the measurability and integrability of the functions $t\mapsto\langle\AAt,z\rangle$ for every $z\in E\widehat{\otimes}_{d_{p'}} F.$

By \cite[Proposition~6.10]{R}, every $z\in E\widehat{\otimes}_{d_{p'}} F$ is of the form $z = \sum_{n=1}^\infty x_n\otimes y_n$ where $(x_n)_{n=1}^\infty\in\ell^w_{p}(E)$ and $(y_n)_{n=1}^\infty\in\ell_{p',}(F).$ Moreover, the series $\sum_{n=1}^\infty x_n\otimes y_n$ converges in the $d_{p'}$ norm. Since $\AAt\in\Pi_p(E,F^*)$ for almost every $t,$ we also have $\langle\AAt,z\rangle = \sum_{n=1}^\infty \langle\AAt, x_n\otimes y_n\rangle = \sum_{n=1}^\infty \langle\AAt x_n,y_n\rangle.$ Finally, by the weak$^*$ measurability of $\AAA,$ scalar functions $t\mapsto\langle\AAt x_n,y_n\rangle$ are measurable, hence $t\mapsto\langle\AAt,z\rangle$ is measurable as the pointwise limit of the measurable functions. In addition, according to the estimate $\zzao{\langle\AAt,z\rangle}\leqslant\sum_{n=1}^\infty \zzao{\langle\AAt x_n,y_n\rangle},$ the integrability of $t\mapsto\langle\AAt,z\rangle$ is a direct consequence of \eqref{bilinuslov}.

To prove the estimate \eqref{geljfnorm}, again by \cite[Proposition~6.10]{R}, for $z\in E\widehat{\otimes}_{d_{p'}} F,$ such that $\|z\|_{E\widehat{\otimes}_{d_{p'}} F} = 1,$ and every $\varepsilon > 0$ there exist $(x_n)_{n=1}^\infty\in\ell_{p}^w(E)$ and $(y_n)_{n=1}^\infty\in\ell_{p',}(F)$ such that $1\leqslant\|(x_n)_{n=1}^\infty\|_{p,w}\cdot\|(y_n)_{n=1}^\infty\|_{p'}<1+\varepsilon.$ Hence, we get

\begin{align}\label{geljfintpom1}
    &\int_\WWW \zzaj{\langle \AAt,z\rangle}d\mu(t) \leqslant \int_\WWW \sum_{n=1}^\infty \zzaj{\langle\AAt x_n,y_n\rangle} d\mu(t) \notag\\
    = &\nnoj{(x_n)_{n=1}^\infty}_{p,w}\cdot\nnoj{(y_n)_{n=1}^\infty}_{p'}
\cdot\!\int_\WWW\! \sum_{n=1}^\infty \zzac{\!\left\langle\!\!\AAt \tfrac{1}{\nnoj{(x_n)_{n=1}^\infty}_{p,w}}x_n,\tfrac{1}{\nnoj{(y_n)_{n=1}^\infty}_{p'}}y_n\!\right\rangle\!} d\mu(t)\notag\\
\leqslant&(1+\varepsilon)\cdot\!\! \int_\WWW\! \sum_{n=1}^\infty \zzaj{\langle\AAt u_n,v_n\rangle} d\mu(t),
\end{align}

where $(u_n)_{n=1}^\iii = \frac{1}{\nnoj{(x_n)_{n=1}^\infty}_{p,w}}(x_n)_{n=1}^\iii$ and $(v_n)_{n=1}^\iii = \frac{1}{\nnoj{(y_n)_{n=1}^\infty}_{p'}}(y_n)_{n=1}^\iii$ are unit vectors in $\lpwE$ and $\ell_{p'}(F)$ respectively. By taking supremum over all unit vectors in the last term of \eqref{geljfintpom1} we get
$$\int_\WWW \zzaj{\langle \AAt,z\rangle}d\mu(t) \leqslant (1+\varepsilon) \cdot \sup_{\substack{\|(u_n)_{n=1}^\infty\|_{p,w}\leqslant 1\\ \|(v_n)_{n=1}^\infty\|_{p'}\leqslant 1}} \int_\WWW\sum_{n=1}^\iii |\langle\AAt u_n,v_n\rangle|d\mu(t),$$
for every unit vector $z$ and every $\varepsilon > 0.$ By taking supremum over $z$ and letting $\varepsilon \to 0,$ we get the desired estimate \eqref{geljfnorm}.
\end{proof}
\begin{remark}
We note that the assumption $\AAt\in\Pi_p(E,F^*)$ is superfluous for the conclusion that $\int_\WWW\AAA d\mu$ belongs to $\Pi_p(E,F^*),$ as can be seen in the proof. On the other hand, it is necessary to conclude that $\AAA\in\GG(\WWW,\mathfrak{M},\mu,\Pi_p(E,F^*)).$ For this reason, the mentioned assumption is omitted from Theorem \ref{bilin}.
\end{remark}
Similarly to \cite[Theorem~2.3]{M24}, the conditions of the previous theorem do not imply the integrability of the function $t\mapsto\pi_{q,p}(\AAt),$ even when it is measurable, as shown in the following example. It is taken from \cite[Example~2.6]{M24}, which is expected, as $\ccq=\Pi_{q,2}(\HH)$ isometrically, for every $q\geqslant2$ (see \cite[Theorem~10.3]{DJT}). Here, $\ccq$ denotes the Schaten--von Neumman ideal, i.e., the ideal of all compact operators $A\in\BH$ such that $\sum_{n=1}^\iii s_n^q(A)<+\iii,$ where $(s_n(A))_{n\in\mathbb{N}}$ is the sequence of singular values of the operator $A.$
\begin{example} We define
$\AAA\colon[0,1)\rightarrow\BH$ by $\AAt\eqd\sum_{n=1}^{+\infty}n\chi_{[\frac{1}{n+1},\frac{1}{n})}(t)e_n\otimes e_1^\star.$ Every $\AAt$ is a finite-rank operator, and hence in $\Pi_{q,2}(\HH).$ Let $(f_m)_{m\inN}\in\ell^w_2(\HH)$ and
$(g_m)_{m\inN}\in\ell_{q'}(\HH)$ be arbitrary.
Now, for every $t\in[\frac{1}{n},\frac{1}{n+1})$ we have that $\sum_{m=1}^{+\iii}|\langle\AAt f_m, g_m\rangle|=n\sum_{m=1}^{+\iii}|\langle f_m,e_1\rangle\langle e_n, g_m\rangle|,$ so

\begin{align*}
  &\int_{[0,1)}\sum_{n=1}^{+\iii}|\langle\AAt f_m, g_m\rangle|dm(t)= \sum_{n=1}^{+\iii}n\int_{\frac{1}{n+1}}^{\frac{1}{n}}\sum_{m=1}^{+\iii}|\langle f_m, e_1\rangle\langle e_n, g_m\rangle|dt\\
  =&\sum_{n=1}^{+\iii}\frac{1}{n+1}\sum_{m=1}^{+\iii}|\langle f_m, e_1\rangle\langle e_n, g_m\rangle|\leqslant
 M\sqrt{\sum_{n=1}^{+\iii}\left(\sum_{m=1}^{+\iii}|\langle f_m, e_1\rangle\langle e_n, g_m\rangle|\right)^2}\\ \leqslant &M\cdot\sum_{m=1}^{+\iii}\sqrt{\sum_{n=1}^{+\iii}|\langle f_m, e_1\rangle|^2\cdot|\langle e_n, g_m\rangle|^2}=
  M\cdot\sum_{m=1}^{+\iii}|\langle f_m, e_1\rangle|\cdot\|g_m\|_{\HH}\\
  \leqslant&M\sqrt[q]{\sum_{m=1}^{+\iii}|\langle f_m, e_1\rangle|^q}\cdot\sqrt[q']{\sum_{m=1}^{+\iii}\|g_m\|_\HH^{q'}}\leqslant M\|(f_m)_{m\inN}\|_{q,w}\cdot\|(g_m)_{m\inN}\|_{q'}<+\iii, 
\end{align*}
where $M=\sqrt{\sum_{n=1}^{+\iii}\frac{1}{(n+1)^2}}=\sqrt{\frac{\pi^2}{6}-1},$ since $\ell^w_2(\HH)\subset\ell^w_q(\HH).$ Thus, the condition \eqref{bilinuslov} holds. Proof that $\int_{[0,1)}\pi_{q,2}(\AAt)dm(t)=+\iii$ is the same as in the mentioned example (as norm in $\ccq$ is a $Q$-norm, since $q\geqslant2).$
\end{example}

Next, we provide an affirmative answer to the question raised in \cite[Remark~2.11]{M24}. First, we borrow some notation from \cite{M24}. Namely, $\Phi$ denotes the symmetric norming (s.n.) function, $\ccc_\Phi(\HH)$ denotes the corresponding ideal, and $B$ denotes the set of all orthonormal systems in the separable Hilbert space $\HH.$ Details can be found in \cite{M24} and the references therein.

Throughout the paper \cite{M24}, the author investigated the consequences of the condition $\int_\WWW \Phi\zzmj{\zzmo{\langle\AAt e_n,f_n\rangle}_{n=1}^\infty} \dt < +\infty$ for every $e,f\in B,$ regarding Gel'fand and Pettis integrability. In Remark~2.11, it was mentioned that it was not verified whether $\sup_{e,f\in B}\int_\WWW \Phi\zzmj{\zzmo{\langle\AAt e_n,f_n\rangle}_{n=1}^\infty} \dt$ is finite under the previous condition. To resolve this question, we first need the following technical lemma regarding the structure of $\ell_2^w (\HH)$ in the case of a separable Hilbert space $\HH.$
\bbl\label{ldvslb}
Let $\HH$ be a separable Hilbert space and $(x_n)_{n=1}^\infty \in\ell_2^w(\HH)$ such that $\nnoo{(x_n)_{n=1}^\infty}_{2,w}\leqslant 1.$ Then, $(x_n)_{n=1}^\infty$ can be written in the form
$$ x_n = f_n + g_n + h_n,$$
where $(f_n)_{n=1}^\infty,(g_n)_{n=1}^\infty$ and $(h_n)_{n=1}^\infty$ are orthonormal bases of $\HH.$
\eel
\begin{proof}
Due to the isometry given by \cite[Proposition~8.2]{DF}, an operator $A\colon\ell_2\to\HH$ defined by $A(e_n) = x_n$ (where $e_n$ is the $n$-th coordinate vector in $\ell_2$) is bounded and $\nnoo{A} = \nnoo{(x_n)_{n=1}^\infty}_{2,w} \leqslant 1.$ According to \cite[Theorem~3.4]{OP86}, as the spaces $\HH$ and $\ell_2$ are isomorphic Hilbert spaces, for the operator $A$ holds $A = U + V + W,$ where $U,V,W\colon\ell_2\to\HH$ are unitary operators and therefore
$$x_n = A e_n = U e_n + V e_n + W e_n.$$
Since the operators $U,V,W$ are unitary, the sequences $\zzmo{Ue_n}_{n=1}^\infty,\zzmo{Ve_n}_{n=1}^\infty$ and $\zzmo{We_n}_{n=1}^\infty$ are orthonormal bases of $\HH.$ 
\end{proof}
Finally, the previous lemma and the Baire Category Theorem establish the finiteness of the mentioned supremum.
\begin{theorem}Let $\HH$ be a separable Hilbert space and let $\AAA\colon\WWW\to\BH$ be weakly$^*$ measurable function such that $\AAt\in\ccc_\Phi(\HH)$ for almost all $t\in\WWW$ and for every orthonormal sequences $(e_n)_{n=1}^\infty$ and $(f_n)_{n=1}^\infty$ holds
\begin{equation}\label{Finormeuslov1}
\int_\WWW \Phi\zzmj{\zzmj{\langle\AAt e_n,f_n\rangle}_{n=1}^\infty} \dt < +\infty.
\end{equation}
Then
\begin{equation}
\sup_{e,f\in B} \int_\WWW \Phi\zzmj{\zzmj{\langle\AAt e_n,f_n\rangle}_{n=1}^\infty} \dt < +\infty.
\end{equation}
\end{theorem}
\begin{proof}
First, we demonstrate that under the assumptions of the Theorem, the integral $\int_\WWW \Phi\zzmo{\langle\AAt x_n,y_n\rangle} \dt$ remains finite for sequences $(x_n)_{n=1}^\infty$ and $(y_n)_{n=1}^\infty$ belonging to $\ell_2^w(\HH).$ Indeed, let $(x_n)_{n=1}^\iii,(y_n)_{n=1}^\infty\in\ell_2^w(\HH)$ such that $\nnoo{(x_n)_{n=1}^\infty}_{2,w}\leqslant 1$ and $\nnoo{(y_n)_{n=1}^\infty}_{2,w}\leqslant 1.$ By the Lemma \ref{ldvslb}, we have the bases $f^i$ and $g^j$ for $i,j=1,2,3$ such that $x_n = f^1_n + f^2_n + f^3_n$ and $y_n = g^1_n + g^2_n + g^3_n.$ Hence $\langle\AAt x_n,y_n\rangle = \sum_{i,j=1}^3 \langle\AAt f^i_n,g^j_n\rangle$ and by the triangle inequality for $\Phi$ holds $\Phi\zzmo{\zzmo{\langle\AAt x_n,y_n\rangle}_{n=1}^\infty}\leqslant \sum_{i,j=1}^3 \Phi\zzmo{\zzmo{\langle\AAt f^i_n,g^j_n\rangle}_{n=1}^\infty}.$ Integration further gives us
\begin{equation}\label{pomocna1}
\int_\WWW\Phi\zzmj{\zzmj{\langle\AAt x_n,y_n\rangle}_{n=1}^\infty} \dt\leqslant \sum_{i,j=1}^3 \int_\WWW\Phi\zzmj{\zzmj{\langle\AAt f^i_n,g^j_n\rangle}_{n=1}^\infty}\dt < +\infty
\end{equation}
where last inequality holds by the \eqref{Finormeuslov1}. By homogeneity, we get the claimed result for all $(x_n)_{n=1}^\iii,(y_n)_{n=1}^\infty\in\ell_2^w(\HH).$

Next, we prove that $\int_\WWW \Phi\zzmo{\zzmo{\langle\AAt x_n,y_n\rangle}_{n=1}^\infty}\dt$ is uniformly bounded for all $(x_n)_{n=1}^\infty,(y_n)_{n=1}^\infty\in\ell_2^w(\HH)$ such that $\nnoo{(x_n)_{n=1}^\infty}_{2,w}, \nnoo{(y_n)_{n=1}^\infty}_{2,w} \leqslant 1.$ Now, let us denote by $Z = \ell_2^w(\HH)\times \ell_2^w(\HH)$ a Banach space endowed with the norm $\|(x,y)\|_Z = \max\{\nnoo{(x_n)_{n=1}^\infty}_{2,w}, \nnoo{(y_n)_{n=1}^\infty}_{2,w}\},$ and for every $M\in\mathbb{N}$ we denote by $Z_M = \{(x,y)\in Z : \int_\WWW \Phi\zzmo{\zzmo{\langle\AAt x_n,y_n\rangle}_{n=1}^\infty}\dt \leqslant M\}.$ By the \eqref{pomocna1} we get $\cup_{M\in\mathbb{N}} Z_M = Z.$ To prove that $Z_M$ is closed, let $(x^m,y^m)\in Z_M,$ for every $m\in\mathbb{N},$ such that $(x^m,y^m)\to(x,y)$ as $m\to\infty.$ As $x^m\to x$ in $\ell_2^w(\HH),$ for every $n\in\mathbb{N}$ and every $h\in\HH,$ such that $\nnoo{h}_\HH\leqslant 1$ we get
$$\zzaj{\langle x^m_n - x_n, h\rangle}\leqslant \zzmt{\sum_{n=1}^\infty \zzaj{\langle x^m_n - x_n, h\rangle}^2}^{1/2} \leqslant \sup_{\nnoo{h}_\HH\leqslant 1} \zzmt{\sum_{n=1}^\infty \zzaj{\langle x^m_n - x_n, h\rangle}^2}^{1/2}.$$
By taking the supremum over $h$ we get $\nnoo{x^m_n - x_n}_\HH \leqslant \nnoo{x^m - x}_{2,w}\to 0.$ Moreover, as operators $\AAt$ are continuous, we also get $\langle\AAt x^m_n,y^m_n\rangle \to \langle\AAt x_n,y_n\rangle,$ as $m\to\infty,$ for every $n\in\mathbb{N}$ and every $t\in\WWW.$ Now, for $N\in\mathbb{N}$ and sequence $(a_n)_{n=1}^\infty$ let us denote by $(a^N_n)_{n=1}^\infty$ the sequence given by
$$a^N\eqd\zzmj{\underbrace{a_1,\ldots,a_N}_N,0,0,\ldots}.$$
According to continuity of $\Phi$ on the space of finite sequences, for every $N\in\mathbb{N}$ holds
$\Phi\zzmo{\zzmo{\langle\AAt x^m_n,y^m_n\rangle^N}_{n=1}^\infty} \to \Phi\zzmo{\zzmo{\langle\AAt x_n,y_n\rangle^N}_{n=1}^\infty},$ and by monotonicity properties of $\Phi$ we finally get
\begin{multline}\label{pomocna2}
\Phi\zzmj{\zzmj{\langle\AAt x_n,y_n\rangle^N}_{n=1}^\infty} = \lim_{m\to\infty} \Phi\zzmj{\zzmj{\langle\AAt x^m_n,y^m_n\rangle^N}_{n=1}^\infty}\\
= \varliminf_{m\to\infty} \Phi\zzmj{\zzmj{\langle\AAt x^m_n,y^m_n\rangle^N}_{n=1}^\infty} \leqslant
\varliminf_{m\to\infty} \Phi\zzmj{\zzmj{\langle\AAt x^m_n,y^m_n\rangle}_{n=1}^\infty}.
\end{multline}
As the last term in \eqref{pomocna2} is independent of $N,$ by taking supremum over $N\in\mathbb{N},$ again due to monotonicity properties of $\Phi$ we get
$$\sup_{N\in\mathbb{N}} \Phi\zzmj{\zzmj{\langle\AAt x_n,y_n\rangle^N}_{n=1}^\infty} = \Phi\zzmj{\zzmj{\langle\AAt x_n,y_n\rangle}_{n=1}^\infty} \leqslant \varliminf_{m\to\infty} \Phi\zzmj{\zzmj{\langle\AAt x^m_n,y^m_n\rangle}_{n=1}^\infty}.$$
Furthermore, by integrating the last inequality and by the Fatou's lemma, follows
\begin{multline}\label{pomocna3}
\int_\WWW \Phi\zzmj{\zzmj{\langle\AAt x_n,y_n\rangle}_{n=1}^\infty} d\mu(t) \leqslant \int_\WWW \varliminf_{m\to\infty} \Phi\zzmj{\zzmj{\langle\AAt x^m_n,y^m_n\rangle}_{n=1}^\infty} d\mu(t)\\ \leqslant \varliminf_{m\to\infty} \int_\WWW \Phi\zzmj{\zzmj{\langle\AAt x^m_n,y^m_n\rangle}_{n=1}^\infty} d\mu(t) \leqslant M
\end{multline}
where the last estimate holds due to $(x^m,y^m)\in Z_M.$ Therefore $(x,y)$ belongs to $Z_M,$ thus $Z_M$ is closed. We note that details and proofs of mentioned properties of s.n. functions can be found in sections 3. and 4. of Chapter III of \cite{gohkr}.

Now, by the Baire's Category Theorem, there exists $M_0$ such that the open ball $B((x^0,y^0),r)\in Z_{M_0}.$ For the ease of notation, let us denote by $S(x,y) = \int_\WWW \Phi\zzmo{\zzmo{\langle\AAt x_n,y_n\rangle}_{n=1}^\infty} d\mu(t)$ and we note that (due to the triangle inequality for $\Phi$) for $S$ holds $S(x+z,t) \leqslant S(x,t) + S(z,t)$ and $S(\alpha x,y) = \zzao{\alpha} S(x,y),$ and similar for the second variable. Therefore, for $(x,y)\in Z$ such that $\nnoo{(x,y)}_Z\leqslant 1,$ we get
\begin{align}
& 4S\zzmj{x,y} = \tfrac{1}{r^2} S\zzmj{rx+ x_0 + rx - x_0,ry+ y_0 + ry - y_0} \leqslant \notag\\
\tfrac{1}{r^2}&\Bigl( S\zzmj{x_0 + rx,y_0 + ry}+S\zzmj{x_0 + rx,y_0 - ry} + \Bigr.\notag\\
&\Bigl. S\zzmj{x_0 - rx,y_0 + ry} + S\zzmj{x_0 - rx,y_0 - ry}\Bigr) \leqslant \tfrac{4}{r^2}\cdot M_0 \label{pomocna4},
\end{align}
where last inequality in \eqref{pomocna4} holds due to $S(x_0\pm rx,y_0\pm ry)\in Z_{M_0}.$ Finally, by Bessel's inequality, every orthonormal sequence $(e_n)_{n=1}^\infty$ belongs to $\ell_2^w(\HH)$ and satisfies $\|(e_n)_{n=1}^\infty\|_{2,w} \leqslant 1.$ Equivalently, $B\times B\subset\{(x,y)\in Z :\|(x,y)\|_Z \leqslant 1\},$ hence
$$\sup_{e,f\in B} \!\int_\WWW \!\!\Phi\zzmo{\zzmo{\langle\AAt e_n,f_n\rangle}_{n=1}^\infty} d\mu(t) \!\leqslant\!\! \sup_{\|(x,y)\|_Z\leqslant 1}\!\int_\WWW \!\!\Phi\zzmo{\zzmo{\langle\AAt x_n,y_n\rangle}_{n=1}^\infty} d\mu(t) \!\leqslant\! \frac{M_0}{r^2} \!<\!\! +\iii.$$
\end{proof}
In the case of positive operators, the conclusions of previous Theorems \ref{bilin} and \ref{Gintuslov} hold under weaker assumptions, at least in one particular case. Due to Lemma \ref{pozneplp}, we can provide alternative sufficient conditions for Gel'fand integrability in the case of operators between continuous and $p$-integrable functions. More precisely,
\begin{theorem}
Let $K$ be a compact space, $(X,\nu)$ be a measure space, $1\leqslant p\leqslant q<+\iii$ and $\AAA\colon\WWW\to\BB(C(K),L_q(X,\nu))$  be a weakly$^*$ integrable o.v. function. If operators $\AAt$ are also positive, then for every measurable $S$ the operator $\int_S \AAA d\mu$ belongs to $\Pi_{q,p}(C(K),L_q(X,\nu)).$ Moreover, in the case $q=p$ holds $\AAA\in\GG(\WWW,\mathfrak{M},\mu,\Pi_p(C(K),L_p(X,\nu)).$
\end{theorem}
\begin{proof}
Due to Proposition \ref{intpoz}, the operator $\int_S \AAA d\mu$ is positive for every $S\in\MM$ and hence, by the Lemma \ref{pozneplp} belongs to $\Pi_{q,p}(C(K),L_q(X,\nu)).$ Moreover, we will prove that $\AAA$ satisfies the condition \eqref{bilinuslov}, implying proclaimed Gel'fand integrability in the case $p=q.$

Let $(f_n)_{n=1}^\iii\in\ell_p^w(C(K))$ and $(\varphi_n)_{n=1}^\iii\in\ell_{q'}(L_{q'}(X,\nu))$ be arbitrary, and let $A:=\int_\WWW\AAA d\mu.$ Since $\nnoo{(f_n)_{n=1}^\infty}_{p,w} = \sup_{s\in K}(\sum_{n=1}^\infty |f_n(s)|^p)^{1/p}$ (see the proof of \cite[Examples~2.9(a)]{DJT}) we note that the sequence $(\zzao{f_n})_{n=1}^\infty$ also belongs to $\ell_{p}^w(C(K)).$ Also, by the definition of the space $\ell_{q'}(L_{q'}(X,\nu)),$ we have that $(|\varphi_n|)_{n=1}^\infty\in\ell_{q'}(L_{q'}(X,\nu)).$ Furthermore, as the operators $\AAt$ are positive and $L_q(X,\nu)$ is Dedekind complete (details can be found in \cite[Example~12.5]{Z}), by \cite[Proposition~2.2.6]{MN} we get $|\AAt f_n| \leqslant \AAt |f_n|,$ for every $n\inN$ and $t\in\WWW.$ Thus,
\begin{align*}
    &\int_\WWW\sum_{n=1}^\iii|\langle\AAt f_n,\varphi_n\rangle|d\mu(t)=\int_\WWW\sum_{n=1}^\iii\left|\int_X(\AAt f_n)(s)\!\cdot\!\varphi(s)d\nu(s)d\mu(t)\right|\\
    \leqslant&\int_\WWW\sum_{n=1}^\iii\int_X\!|\AAt f_n|(s)\!\cdot\!|\varphi_n|(s)d\nu(s)d\mu(t)\leqslant\sum_{n=1}^\infty\int_\WWW\int_X \!(\AAt |f_n|)(s)\!\cdot\!|\varphi_n|(s) d\nu(s)d\mu(t)\\
    =&\sum_{n=1}^\iii\int_\WWW\langle\AAt|f_n|,|\varphi_n|\rangle d\mu(t)=\sum_{n=1}^\iii\langle A|f_n|,|\varphi_n|\rangle<+\iii,
\end{align*}
due to Lemma \ref{qpekv2}, as operator $A$ is $(q,p)$-summing, being positive.

\end{proof}
Without the assumption of positivity, the conclusions of the previous Theorem do not hold. More precisely, condition \eqref{bilinuslov} needs not to be satisfied, at least for some appropriate compact set $K$ and measure space $(X,\nu),$ as the following example shows.
\begin{example}
Let $\WWW = \mathbb{N}$ with $\mu$ being the counting measure. By the Lemma \ref{merljivost} function $\AAA\colon\WWW\to\BB(E,F^*)$ is weak$^*$ integrable if
$$\sum_{n=1}^\infty \zzaj{\langle\AAA_n x,y\rangle} < +\infty$$
for every $x\in E$ and $y\in F.$ For $1\leqslant p\leqslant+\iii,$ let $T\in\BB(E,\ell_p)$ and let us denote  $P_n\colon\ell_p\to\ell_p$ an operator defined by
$$P_n (x_1,\ldots,x_{n-1},x_n,x_{n+1},\ldots) = (0,\ldots,0,x_n,0,\ldots),$$
for every $n\inN.$ Now, we define an o.v. function by $\AAA_n = P_n T.$ Then for $x\in E$ and $y\in\ell_{p'}$ (it is enough to take $y\in c_0$ in the case $p=1$) holds
\begin{multline}
\sum_{n=1}^\infty \zzaj{\langle \AAA_n x,y\rangle} = \sum_{n=1}^\infty \zzat{\sum_{m=1}^\infty \zzmo{P_n Tx}_m y_m} = \sum_{n=1}^\infty\zzaj{\zzmo{ Tx}_n y_n}\\
\leqslant \sqrt[p]{\sum_{n=1}^\infty\zzaj{\zzmo{ Tx}_n}^p} \cdot\sqrt[p']{\sum_{n=1}^\infty \zzaj{y_n}^{p'}} = \nnoj{Tx}_p \cdot\nnoj{y}_{p'}.
\end{multline}
Therefore $\AAA$ is weak$^*$ integrable, so for every $x\in E$ and $y\in\ell_{p'}$ holds
$$\left\langle\left(\int_\mathbb{N}\AAA d\mu\right) x,y\right\rangle=\int_{\mathbb{N}}\langle\AAA_nx,y\rangle d\mu(n)=\sum_{n=1}^\iii\langle P_nTx,y\rangle=\langle Tx,y\rangle,$$
where the last equality holds because $\sum_{n=1}^\infty P_n T$ converges in weak$^*$ topology to $T.$ Thus, $\int_\mathbb{N}\AAA d\mu=T.$

Next, for every $1\leqslant p<2$ (see \cite[Remark~3.8(a)]{DJT}) and $2<p<+\iii$ (see \cite[Theorem~7]{KW}), there exists $T\in\BB(C(K),\ell_p)$ (for appropriate compact $K$) that is not $p$-summing. For such $T,$ the previously defined o.v. function $n\mapsto\AAA_n = P_n T$ cannot satisfy condition \eqref{bilinuslov}. Moreover, we note that $P_n T$ is $p$-summing for every $n\in\mathbb{N}$ being a rank-one operator.

\end{example}

\section*{Statements and Declarations}

\noindent\textbf{Funding} The authors have been funded by an MPNTR grant, No.
174017, Serbia.\\

\noindent\textbf{Disclosure statement} On behalf of all authors, the corresponding author states that there is no conflict of interest.\\

\noindent\textbf{Authors' contributions} All authors have contributed equally to writing and preparing the manuscript text.

\end{document}